\documentclass[12pt]{article}
\usepackage{fullpage}
\usepackage[utf8]{inputenc}
\usepackage[russian,english]{babel}
\usepackage{amsmath, amssymb, amsfonts, amsthm, mathtools, xcolor}
\usepackage{tikz}
\usepackage{scrextend}
\usepackage{url}
\usepackage{bm}
\usepackage[T2A]{fontenc}
\usepackage{array,tabularx,tabulary,booktabs}
\usepackage{euscript, mathrsfs}	

\newcommand{\setbuilder}[2]{\left\{#1\ \colon #2\right\}}

\usepackage{graphicx}

\usepackage{graphicx, wrapfig}
\usepackage{epigraph, xcolor, hyperref}
\usepackage{array,tabularx,tabulary,booktabs}
\usepackage{euscript, mathrsfs}	
\usepackage{tikz}
\usetikzlibrary{patterns,decorations.text,positioning,arrows,shapes,decorations.markings}

\newcounter{iii}

\newcommand*\phantomrel[1]{\mathrel{\phantom{#1}}}

\newcommand{\R}{{\mathbb R}}

\newcommand{\eps}{{\varepsilon}}

\theoremstyle{plain}
\newtheorem{thm}{Theorem}

\newtheorem{lem}{Lemma}
\newtheorem{cla}{Claim}
\newtheorem{prop}{Proposition}
\newtheorem{pro}{Problem}

\newtheorem{obs}{Observation}

\theoremstyle{definition}
\newtheorem{defi}{Definition}

\newtheorem{remark}{Remark}

\date{}

\title{Almost sharp bounds on the number of discrete chains in the plane}
\author{N\'ora Frankl\thanks{London School of Economics, UK and Moscow Institute of Physics and Technology, Russia. Email: {\tt n.frankl@lse.ac.uk}. Supported by an LMS Early Career Fellowship. Research was partially supported by the National Research, Development,
and Innovation Office, NKFIH Grant K119670.} \and Andrey Kupavskii\thanks{G-SCOP, CNRS, University Grenoble-Alpes, France; Moscow Institute of Physics and Technology, Russia; Email: {\tt kupavskii@yandex.ru}}}
\begin{document}
\maketitle

\begin{abstract}
The following generalisation of the Erdős Unit Distance problem was recently suggested by Palsson, Senger and Sheffer. For a fixed sequence $\bm\delta=(\delta_1,\dots,\delta_k)$ of $k$ distances, a $(k+1)$-tuple $(p_1,\dots,p_{k+1})$ of distinct points in $\mathbb{R}^d$ is called a \emph{$k$-chain} if $\|p_j-p_{j+1}\| = \delta_j$ for every $1\leq j \leq k$. What is the maximum number $C_k^d(n)$ of $k$-chains in a set of $n$ points in $\mathbb{R}^d$?
Improving the results of Palsson, Senger and Sheffer, we essentially determine this maximum for all $k$ in the planar case.
It is only  for $k\equiv 1$ (mod $3$) that the answer depends
on the maximum number of unit distances in a set of $n$ points. We also obtain almost sharp results for even $k$ in dimension $3$, and propose further generalisations.
 \end{abstract}

\section{Introduction}\label{sec:introduction}

Determining the maximum possible number of pairs $u_d(n)$ at distance $1$ apart in a set of $n$ points in $\mathbb{R}^d$ for $d=2$ is one of the  central questions in combinatorial geometry, known as the Erdős Unit Distance problem. The  question dates back to 1946, and despite much effort, the best known upper and lower bounds are still very far apart. For some constants $C,c>0$, we have
\[n^{1+c/\log\log n}\leq u_2(n)\le Cn^{4/3},\]
where the lower bound is due to Erdős \cite{erdos} and the upper bound is due to Spencer, Szemerédi and Trotter \cite{SST}. Recently, there has been great progress in a closely related
problem of determining the minimum number of distinct distances between $n$ points on the plane due to Guth and Katz \cite{GK}, but the powerful algebraic machinery they used has not yet given any improvement for the unit distance question.

As in the planar case, the best known upper and lower bounds in the $3$-dimensional case are also far apart. 
For every $\varepsilon>0$ there are constants $c,C>0$ such that we have

\begin{equation}\label{eqzahl}
cn^{4/3}\log\log n\le u_3(n)\le Cn^{295/197+\varepsilon},
\end{equation}
where the lower bound is due to Erdős \cite{erdos3}, and the upper bound is due to Zahl \cite{Za2}. The latter is a recent improvement upon the   upper bound $O(n^{3/2})$ by Kaplan, Matou\v sek, Safernová, and Sharir \cite{KMSS}, and Zahl \cite{Za}. 

This paper can be seen as an effort to find generalisations of the Unit Distance problem that are within the  reach of our current methods. In what follows, we describe the generalisation that we work with.

Palsson, Senger and Sheffer \cite{Shef} suggested the following question.
Let $\bm\delta=(\delta_1,\dots,\delta_k)$ be a fixed sequence of $k$ positive reals. A $(k+1)$-tuple $(p_1,\dots,p_{k+1})$ of distinct points in $\mathbb{R}^d$ is called a \emph{$k$-chain} if $\|p_i-p_{i+1}\|=\delta_i$ for all $i=1,\dots,k$. For every fixed $k$ determine $C^d_k(n)$, the maximum number of $k$-chains that can be spanned by a set of $n$ points in $\mathbb{R}^d$. We do not include $\bm\delta$ in the notation, since our results, with the exception of Proposition \ref{3dbound} do not depend on $\bm\delta$ up to the order of magnitude.
The authors of \cite{Shef} give the following lower bound on $C^2_k(n)$: 
\[C^2_k(n)=\Omega\left (n^{\lfloor (k+1)/3 \rfloor+1}\right ). \]
They also provided upper bounds in terms of the maximum number of unit distances.

\begin{prop}[Palsson, Senger, and Sheffer \cite{Shef}] 
\begin{equation*} C^2_k(n) =
    \begin{cases}
      O\left (n\cdot u_2(n)^{k/3}   \right ) & \text{\rm if $k\equiv 0$ (mod $3$),}\\
      O\left (u_2(n)^{(k+2)/3}\right ) & \text{\rm if $k\equiv 1$ (mod $3$),}\\
      O\left (n^2\cdot u_2(n)^{(k-2)/3}\right ) & \text{\rm if $k\equiv 2$ (mod $3$).}
    \end{cases}       
\end{equation*}
\end{prop}

If $u_2(n) = O(n^{1+\varepsilon})$ for any $
\eps>0$, which is conjectured to hold, then the upper bounds in
the proposition above almost match the lower bound given above.
However, as we have already mentioned, determining the order of magnitude of $u_2(n)$ has proved to be a very hard problem and is very far from its resolution.
Thus, it is interesting to obtain ``unconditional'' bounds, that depend on the value of $u_2(n)$ as little as possible.  In \cite{Shef}, the following ``unconditional'' upper bounds were proved in the planar case.

\begin{thm}[Palsson, Senger, and Sheffer \cite{Shef}]\label{PSS2} $C_2^2(n)=\Theta(n^2)$, and for every $k\geq 3$ we have
\[C^2_k(n)=O\left (n^{2k/5+1+\gamma_k}\right ), \]
where $\gamma_k\leq \frac{1}{12}$, and $\gamma_k \to \frac{4}{75}$ as $k\to \infty$.
\end{thm}

In our main result, in two-third of the cases we almost determine the value of $C^2_k(n),$ no matter what the value of $u_2(n)$ is, by matching the lower bounds given in Theorem~\ref{PSS2}.
Further, we show that in the remaining cases determining $C^2_k(n)$ essentially reduces to determining the maximum number of unit distances.

\begin{thm}\label{main}For every integer $k\geq 1$ we have\,\footnote{In what follows $f(n)= \tilde O(g(n))$ means that there exist positive constants $c, C$ such that $ f(n)/g(n) \le C\log^{c} n$ for every sufficiently large $n.$ We write $f(n) = \tilde{\Omega}(g(n))$ if $g(n) = \tilde O(f(n))$, and $f(n) = \tilde{\Theta}(g(n))$ if $f(n) = \tilde O(g(n))$ and $g(n) = \tilde O(f(n)).$}
\[C^2_k(n)=\tilde{\Theta}\left (n^{\lfloor (k+1)/3 \rfloor+1}\right ) \textrm{ if }  k\equiv 0,2 \text{ \rm (mod } 3 \text{\rm )},\]
and for any $\varepsilon> 0$ we have
\[C^2_k(n)=\Omega\left (n^{(k-1)/3}u_2(n)\right) 
\textrm{ and } C^2_k(n)=O\left(n^{(k-1)/3+\varepsilon}u_2(n)\right)  \textrm { if } k\equiv 1 \text{ \rm(mod $3$)}. \]
\end{thm}

Let us turn our attention to the  $3$-dimensional case. The following was proved in \cite{Shef}.

\begin{thm}[Palsson, Senger, and Sheffer \cite{Shef}]\label{thmshefr3} For any integer $k\geq 2$, we have
\[C^3_k(n)=\Omega \left (n^{\lfloor k/2 \rfloor +1}\right ),\]
and 
\begin{equation*}
  C^3_k(n) =
    \begin{cases}
      O\left (n^{2k/3+1}\right ) & \text{\rm if $k\equiv 0$ (mod $3$),}\\
      O\left (n^{2k/3+23/33+\varepsilon}\right ) & \text{\rm if $k\equiv 1$ (mod $3$),}\\
      O\left (n^{2k/3+2/3}\right ) & \text{\rm if $k\equiv 2$ (mod $3$).}
    \end{cases}       
\end{equation*}
\end{thm}
We improve their upper bound and essentially settle the problem for even $k$.
\begin{thm}\label{3d}For any integer $k\geq 2$ we have
\[C^3_k(n)=\tilde O\left (n^{k/2+1}\right).\] Furthermore, for even $k$ we have \[C^3_k(n)=\tilde{\Theta}\left (n^{k/2+1}\right).\]
\end{thm}

We also improve the lower bound from Theorem~\ref{thmshefr3} for odd $k$ and $\bm\delta=(1,\dots,1)$. Let $us_3(n)$ be the maximum possible number of pairs at unit distance apart in $X\times Y$, where $X$ is a set of $n$ points in $\mathbb{R}^3$ and $Y$ is a set of $n$ points on a sphere in $\mathbb{R}^3$.

\begin{prop}\label{oddk}Let $k\geq 3$ odd. Then for $\bm\delta=(1,\dots,1)$ we have
\[C^3_k(n)=\Omega\left (\max\left \{\frac{u_3(n)^k}{n^{k-1}},us_3(n)n^{(k-1)/2}\right \}\right ).\]
\end{prop}

By using stereographic projection we obtain that $us_3(n)$ equals the maximum number of incidences between a set of $n$ points and a set of $n$ circles (not necessarily of the same radii) in the plane. Thus we have
\[cn^{4/3}\leq us_3(n)=\tilde{O}\left ( n^{15/11} \right )\]
(For the lower bound see \cite{us1}, and for the upper bound see \cite{us2,AS,us3}).)
Therefore, in general we cannot tell which of the two bounds in Proposition~\ref{oddk} is better. However, for large $k$ the second term is larger than the first due to \eqref{eqzahl}.

Finally, we note that
for $d\geq 4$ we have $C_k^d(n)=\Theta(n^{k+1})$. Indeed, we clearly have $C_k^d(n)=O(n^{k+1})$. To see that $C_k^d(n)=\Omega(n^{k+1})$, take two orthogonal circles of radius $1/\sqrt{2}$ centred at the origin and choose $n/2$ points on each of them. Then any sequence of $k+1$ points that alternate between the two circles forms a path in which all edges have unit length. The exact value of $u_d(n)$ for large $n$ and even $d\geq 4$ was determined by Brass \cite{brass} ($d=4)$ and Swanepoel \cite{Sw} ($d\geq 6$), by using stability results form extremal graph theory.

\section{Preliminaries}

We denote by $u_d(m,n)$ the maximum number of incidences between a set of $m$ points and $n$ spheres\footnote{circles, if $d=2$} of fixed radius in $\mathbb{R}^d$. In other words, $u_d(m,n)$ is the maximum number of red-blue pairs spanning a given distance in a set of $m$ red and $n$ blue points in $\mathbb{R}^d$. By the result of Spencer, Szemerédi and Trotter \cite{SST},  we have
\begin{equation}\label{2drich}
    u_2(m,n)=O\left (m^{\frac{2}{3}}n^{\frac{2}{3}}+m+n\right).
\end{equation}

For given $r$ and $\delta$ we say that a point $p$ is  \emph{$r$-rich} with respect to a set $P\subseteq \mathbb{R}^d$ and to a distance $\delta$, if the sphere of radius $\delta$ around $p$ contains at least $r$ points of $P$. If $P\subseteq \mathbb{R}^2$ and $|P|=n^x$, then \eqref{2drich} implies that the number of points that are $n^{\alpha}$-rich with respect to $P$ and to a given distance $\delta$ is
\begin{equation}\label{2drichness}
    O\left (n^{2x-3\alpha}+n^{x-\alpha}\right ).
\end{equation}

The bound
\begin{equation}\label{eq3d}
    u_3(m,n)=O\left (m^{\frac3{4}}n^{\frac34}+m+n\right)
\end{equation}
is due to Zahl \cite{Za2} and Kaplan, Matou\v sek, Safernov\'a, and Sharir \cite{KMSS}.
It implies that for $P\subseteq \mathbb{R}^3$ with $|P|=n^x$ the number of points that are $n^{\alpha}$-rich with respect to $P$ and to a given distance $\delta$ is
\begin{equation}\label{richness}
    O\left (n^{3x-4\alpha}+n^{x-\alpha}\right ).
\end{equation}

\section{Bounds in \texorpdfstring{$\bm{\mathbb{R}^2}$}{Lg}}\label{sec3}

For a fixed $\bm\delta=(\delta_1,\dots,\delta_k)$ and $P_1\dots,P_{k+1}\subseteq \mathbb{R}^2$ we denote by $\mathcal{C}_k(P_1,\dots,P_{k+1})$ the family of $(k+1)$-tuples $(p_1,\dots,p_{k+1})$ with $p_i\in P_i$ for all $i\in[k+1]$, $\|p_i-p_{i+1}\|=\delta_i$ for all $i\in[k]$ and with $p_i\neq p_j$ for $i\neq j$.
Let $C_k(P_1,\dots,P_{k+1})=|\mathcal{C}_k(P_1,\dots,P_{k+1})|$ and 
\[C_k(n_1,\dots,n_{k+1})=\max C_k(P_1,\dots,P_{k+1}),\]
where the maximum is taken over all sets sets $P_1,\ldots, P_{k+1}$ subject to $|P_i|\le n_i$ for all $i\in [k+1]$.

We have $C^2_k(n)\leq C_k(n,\dots,n)\leq C^2_k\left((k+1)n\right)$. Indeed, for the lower bound choose $P_i=P$ for every $1\leq i \leq k+1$, and for the upper bound note that $|P_1\cup \dots \cup P_{k+1}|\leq (k+1)n$. Since we are only interested in the order of magnitude of $C^2_k(n)$ for fixed $k$, we are going to bound $C_k(n,\dots,n)$ instead of $C^2_k(n)$. 

In Section~\ref{sec31}, we are going to prove the lower bounds from Theorem~\ref{main}. In Section~\ref{sec32}, we are going to prove an upper bound on $C_k(n,\dots,n)$, which is almost tight for $k\equiv 0,2$ (mod $3$). The case  $k\equiv 1$ (mod $3$) is significantly more complicated. We will explain the case $k=4$ case separately in Section~\ref{sec33}, and then the general case in Section~\ref{sec34}.

\subsection{Lower bounds}\label{sec31}
For completeness, we present constructions for all congruence classes modulo $3$. For $k\equiv 0,2$ they were described in \cite{Shef}.

\begin{prop}\label{proplow} 
For any fixed distance-vector $\bm\delta = (\delta_1,\ldots,\delta_k)$ of positive reals we have $$C_k(n,\ldots,n)= \begin{cases} \Omega(n^{\lfloor (k+1)/3\rfloor +1}),& \mbox{if } k \equiv 0,2({\rm mod\ } 3),\\
\Omega(n^{(k-1)/3}\cdot u_2(n)), & \mbox{if } k \equiv 1({\rm mod\ } 3).
\end{cases}$$
\end{prop}
In the proof of the proposition, we shall need the following proposition, using which was suggested to us by D\"om\"ot\"or P\'alv\"olgyi.
\begin{prop}\label{Propprob}Fix $\varepsilon>0$. Then there exists $\gamma = \gamma(\varepsilon)>0$, such that for any $n$ there exist two sets $X_1,X_2$ of points on the plane, $|X_1|,|X_2|\le n$, such that:
\begin{itemize}
    \item[(i)] The diameter of $X_2$ is at most $\varepsilon;$
    \item[(ii)] The number of unit distances between $X_1$ and $X_2$ is at least $\gamma u_2(n,n).$
\end{itemize}
\end{prop}
\begin{proof}
Take sets $Z_1,Z_2$ of $n$ points on the plane each, such that the number of unit distances between them is $u_2(n,n)$. Consider a bipartite graph $G = (Z_1\cup Z_2,E)$, where edges connect vertices at unit distance apart. Take an infinite grid formed by lines $y = 10i+\eta_1, x = 10j+\eta_2,$ where $i,j\in \mathbb Z$ and $\eta_1,\eta_2$ are chosen from $[0,10)$ uniformly at random. Then the expected number of edges of $G$ that are `cut' by a line in the grid is at most $\frac 12 u_2(n,n)$ and, therefore, putting $G' = G'(\eta_1,\eta_2) = (V, E')$ to be a subgraph of $G$ formed by all edges that are not `cut' by the grid (i.e., which endpoints lie inside the same square of the grid), there exist a choice of $\eta_1$, $\eta_2$ such that $|E'|\ge \frac 12 |E|$ and, moreover, no vertex of $G$ lies on a line of the grid. 

For a square $S$ of the grid, let $G'[S]$ be a subgraph of $G'$ induced on the vertices lying in $S$. Note that $G'$ is a disjoint union of $G'[S]$ over all possible $S$. Now translate all vertices in $G'[S]$ by an appropriate vector, so that a) all of them lie within a square $[0,11]\times [0,11]$ and b) no vertices from different translates coincide. Denote the new graph $G'' = (Y_1\cup Y_2,E'')$, where $Y_i$ is the union of translates of vertices from $Z_i$, $i=1,2$. It is clear that $|V(G'')| = |V(G')|$ and $|E''| = |E(G')|\ge \frac 12 u_2(n,n)$. Moreover, $V(G'')$ lies inside the square $[0,11]\times [0,11]$.

Put $X_1:=Y_1$. Next, partition $[0,11]$ into $O(\varepsilon^2)$ squares of side at most $\varepsilon/2$ and choose a square $C$ out of them such that the number of edges from $G''$ emanating from vertices in $X_2:=Y_2\cap C$ is maximal. We claim that $X_1,X_2$ as above satisfy the conditions of the proposition. First, it is easy to see that the diameter of $X_2$ is at most $\varepsilon.$  Second, by the choice of $C,$ the number of edges from $G''$ between $X_1$ and $X_2$ is at least $\Omega(\varepsilon^2 |E''|) = \Omega(\varepsilon^2 u_2(n,n)),$ as desired.
\end{proof}

\begin{proof}[Proof of Proposition~\ref{proplow}] We prove the following slightly stronger statement by induction on $k$. For any fixed $\varepsilon>0$ the lower bound claimed in the proposition  can be achieved on sets $P_1,\ldots, P_{k+1}$, such that the diameter of $P_{k+1}$ is at most $\varepsilon$ (with $\Omega$ depending on $\varepsilon$). 

First, we show the base of induction ($k=0,1,2$). Note that  $C_0(n) = n$ and the set $P_1$ can be chosen to have diameter at most $\varepsilon.$ For $k=1$, we can use the construction from Proposition~\ref{proplow}. For $k=2$, let $P_2 = \{x\}$ for some point $x$, and let $P_1$, $P_3$ be disjoint sets of $n$ points on arcs of length $\varepsilon$ lying on the circles around $x$ of radii $\delta_1$ and $\delta_3$, respectively. Then we have that $C_2(P_1,P_2,P_3)=n^2$. 

Next, fix a vector $\bm\delta = (\delta_1,\ldots,\delta_{k+3}),$ put $\varepsilon = \min \{ \delta_{k+1}/3,\delta_{k+2}/3\}$ and let $\bm\delta' = (\delta_1,\ldots,\delta_k)$. We apply the inductive statement with $\varepsilon$ and $\bm \delta '$, and obtain the corresponding sets $P_1,\ldots, P_{k+1}.$ Put $P_{k+3} = \{x\},$ where $x$ is an arbitrary point on the plane that is at distance  $\delta_{k+1}+\delta_{k+2}-2\varepsilon$ from some point $y$ in $P_{k+1}$.  Put $P_{k+2}$ to be a set of $|P_{k+1}|$ points on the circle $C$ of radius $\delta_{k+2}$ around $x,$ such that for each point in $P_{k+1}$ there is at least one point in $P_{k+2}$ at distance $\delta_{k+1}$. For each point $z$ in $P_{k+1}$ it is possible to choose the corresponding point for $P_{k+2}$ since the minimum of the distance between a point on the circle  $C$ and $z$ is at most $\|x-y\|-\delta_{k+2}+{\rm diam}(P_{k+1}) = \delta_{k+1}-\varepsilon<\delta_{k+1}$, while the maximum distance is at least $\|x-y\|+\delta_{k+2}-{\rm diam}(P_{k+1}) =\delta_{k+1}+2\delta_{k+2}-3\varepsilon>\delta_{k+1}$. (Note that for bounding both the maximum and the minimum distances  we used triangle inequality and the fact that $P_{k+1}$ has diameter at most $\varepsilon.$) We use flexibility in the choice of $x$ to assure that all additional points are different from the points in $P_1,\ldots, P_{k+1}$.

Finally, put $P_{k+4}$ to be a set of $n$ points on a sufficiently small arc on the circle of radius $\delta_{k+3}$ around $x.$

Since by construction every $k$-chain from $P_1\times \dots \times P_{k+1}$ can be extended to a $k+3$ chain in $P_1\times \dots \times P_{k+1}$ in at least $n$ different ways, we obtain that $C_2(P_1,\ldots, P_{k+4}) \ge n C_2(P_1,\ldots, P_{k+1})$. Further, $P_{k+4}$ can be chosen to satisfy any fixed requirement on the diameter.
\end{proof}

\subsection{Upper bound for the \texorpdfstring{$\bm{{k\equiv} 0,2}$}{Lg} (mod \texorpdfstring{$\bm{3}$}{Lg}) cases}\label{sec32}
We fix  $\bm\delta=(\delta_1,\dots,\delta_k)$ throughout the remainder of Section~\ref{sec3}. All logs are base $2$.

\begin{thm}\label{firstbound} For any fixed integer $k\geq 0$ and $x,y\in [0,1]$, we have \begin{equation*}\label{eqmain} C_k(n^{x},n,\ldots,n,n^{y})=\tilde O\left (
  n^{\frac {f(k)+x+y}{3}} \right ),
\end{equation*} where $f(k) = k+2$ if $k\equiv 2$ \emph{(mod $3$)} and $f(k) = k+1$ otherwise. 
\end{thm}

Theorem \ref{firstbound} implies the upper bounds in Theorem \ref{main} for $k\equiv 0,2$ (mod $3$) by taking $x=y=1$. It is easier, however, to prove this more general statement than the upper bounds in Theorem~\ref{main} directly. Having varied sizes of the first and the last groups of points allows for a seamless use of induction.

\begin{proof}[Proof of Theorem \ref{firstbound}] The proof is by induction on $k$. Let us first verify the statement for $k\le 2.$ (Note that, for $k=0$, we should have $x=y$.) We have 
\begin{align}
C_0(n^x)&\le n^x = O\left(n^{\frac {1+x+y}3}\right), \nonumber\\
\label{k=1}
C_1(n^x,n^y)&\leq u_2(n^x,n^y)=O\left (n^{\frac{2}{3}(x+y)}+n^x+n^y\right )=O\left (n^{\frac{2+x+y}3}\right),\\
\label{ktwo} 
C_2(n^x,n,n^y)&\leq 2 n^xn^y= O\left (n^{\frac{4+x+y}{3}}\right ),
\end{align}
where \eqref{k=1} follows from \eqref{2drich} and \eqref{ktwo} follows from the fact that each pair $(p_1,p_3)$ can be extended to a $2$-chain $(p_1,p_2,p_3)$ in at most $2$ different ways.

Next, let $k\geq 3$. Take $P_1,\dots,P_{k+1}\subseteq \mathbb{R}^2$  with $|P_1|=n^x$, $|P_{k+1}|=n^y$, and $|P_i|=n$ for $2\leq i \leq k$. Denote by $P_2^{\alpha}\subseteq P_2$ the set of those points in $P_2$ that are at least $n^{\alpha}$-rich but at most $2n^{\alpha}$-rich with respect to $P_1$ and $\delta_1$. Similarly, we denote by $P_k^{\beta}\subseteq P_k$ the set of those points in $P_k$ that are at least $n^{\beta}$-rich but at most $2n^{\beta}$-rich with respect to $P_{k+1}$ and $\delta_k$.

Applying a standard dyadic decomposition argument twice implies that
\[\mathcal{C}_k(P_1,P_2\dots,P_k,P_{k+1})= \bigcup_{\alpha,\beta} \mathcal{C}_k(P_1,P_2^{\alpha},P_3,\dots, P_{k-1},P_k^{\beta},P_{k+1}),\]
where the union is taken over all $\alpha,\beta\in \left \{\frac{i}{\log n}: i=0,\ldots, \lceil\log n\rceil\right\}$. Since the cardinality of the latter set is at most $\log n+2$, it is sufficient to prove that for every $\alpha$ and $\beta$ we have 
\begin{equation}\label{alphabeta}
C_k(P_1,P_2^{\alpha},P_3,\dots,P_{k-1},P_k^{\beta},P_{k+1})= \tilde O\left ( n^{\frac{f(k)+x+y}{3}}\right ).
\end{equation}

To prove this, we consider three cases.

\bigskip

{\bf Case 1:  $\bm{\alpha\geq \frac{x}{2}}$.} By \eqref{2drichness} we have $|P_2^{\alpha}|=O(n^{x-\alpha})$. Therefore the number of pairs \mbox{$(p_1,p_2)\in P_1\times P_2^{\alpha}$} with $\|p_1-p_2\|=\delta_1$ is at most $O(n^x)$. Since every pair $(p_1,p_2)\in P_1\times P_2^{\alpha}$ and every $(k-3)$-chain $(p_4,\dots,p_{k+1})\in P_4\times\dots \times P_{k}^{\beta}\times P_{k+1}$ can be extended to a $k$-chain $(p_1,\dots,p_{k+1})\in P_1\times \dots\times P_{k+1}$ in at most two different ways, we obtain
\[C_k(P_1,P_2^{\alpha},\dots,P_k^{\beta},P_{k+1})\leq 4O(n^x)C_{k-3}(P_4,\dots,P_k^{\beta},P_{k+1}).\]

By induction we have 
 \[C_{k-3}(P_4,\dots,P_k^{\beta},P_{k+1})= \tilde O\left  ( n^{\frac{f(k-3)+1+y}{3}}\right ).\]
 These two displayed formulas and the fact that $f(k-3)=f(k)-3$ imply \eqref{alphabeta}. 
\bigskip

{\bf Case 2:  $\bm{\beta \geq \frac{y}{2}}$.} By symmetry, this case can be treated in the same way as Case 1.

\bigskip

{\bf Case 3:  $\bm{\alpha \leq \frac{x}{2}}$ and $\bm{\beta\leq \frac{y}{2}}$.} By \eqref{2drichness} we have $|P_2^{\alpha}|= O\left (n^{2x-3\alpha}\right )$ and $|P_k^{\beta}|= O\left 
(n^{2y-3\beta}\right)$. The number of $(k-2)$-chains in $P_2^{\alpha}\times P_3\times\dots\times P_{k-1}\times P_k^{\beta}$ is $C_{k-2}(P_2^{\alpha},P_3,\dots,P_{k-1}, P_k^{\beta})$, and every $(k-2)$-chain 
$(p_2,\dots,p_k)\in P_2^{\alpha}\times P_3\times\dots\times P_{k-1}\times P_k^{\beta}$ can be extended at 
most $4n^{\alpha+\beta}$ ways to a $k$-chain in $P_1\times P_2^{\alpha}\times \dots\times P_k^{\beta}\times P_{k+1}$. Thus 
\[C_{k}(P_1,P_2^{\alpha},\dots,P_k^{\beta},P_{k+1})\leq 4n^{\alpha+\beta}C_{k-2}(P_2^{\alpha},\dots,P_k^{\beta}).\]

By induction we have
\[C_{k-2}(P_2^{\alpha},\dots,P_k^{\beta})= \tilde O\left (n^{\frac{f(k-2)+2x-3\alpha+2y-3\beta}{3}}\right ).\]
For $k \equiv 0,2$ (mod $3$) we have $f(k)\ge f(k-2)+2,$ and thus 
\begin{multline*}
C_{k}(P_1,P_2^{\alpha},\dots,P_k^{\beta},P_{k+1})= \tilde O\left(n^{\alpha+\beta}n^{\frac{f(k-2)+2x-3\alpha+2y-3\beta}{3}}\right)\\[4pt] = \tilde O\left(n^{\frac{f(k)-2+2x+2y}{3}}\right) = \tilde O\left(n^{\frac{f(k)+x+y}{3}}\right).
\end{multline*}

If $k\equiv 1$ (mod $3$) then $f(k)<f(k-2)+2$, and thus the argument above does not work. However, we then have $f(k)= f(k-1)+1$, and we can use the bound
\[C_{k}(P_1,P_2^{\alpha},\dots,P_k^{\beta},P_{k+1})\leq 2n^{\alpha}C_{k-1}(P_2^{\alpha},P_3,\dots,P_{k+1}),\]
obtained in an analogous way. This gives
\[C_{k}(P_1,P_2^{\alpha},P_3,\dots,P_{k+1})= \tilde O\left(n^{\alpha}n^{\frac{f(k-1)+2x-3\alpha+y}{3}}\right) = \tilde O\left(n^{\frac{f(k)-1+2x+y}{3}}\right) = \tilde O\left(n^{\frac{f(k)+x+y}{3}}\right).\]
\end{proof}

\begin{remark} The proof above is not sufficient to obtain an almost sharp bound in the $k \equiv 1$ (mod $3$) case for two reasons. First, for these $k$ any analogue of Theorem \ref{firstbound} would involve taking maximums of two expressions, where one contains $u_2(n^x,n)$ and the other contains $u_2(n^y,n)$. However, due to our lack of good understanding of how $u_2(n^x,n)$ changes as $x$ is increasing, this is difficult to work with.

Second, on a more technical side, while Case 1 and  Case 2 in the above proof would go through with any reasonable inductive statement, Case 3 would fail. The main reason for this is that $C_k$ as a function of $k$ makes jumps at every third value of $k$, and remains essentially the same, or changes by $u_2(n,n)/n$ for the other values of $k$. Thus one would need to remove three vertices from the path to make the induction work. However, the path has only two ends, and removing vertices other than the endpoints turns out to be intractable.
\end{remark}

\subsection{Upper bound for \texorpdfstring{$\bm{k=4}$}{Lg}}\label{sec33}
In this section we prove the upper bound in Theorem~\ref{main} for $k=4$. Let $P_1,\dots, P_5$ be five sets of $n$ points. We will show that $C_4(P_1,\dots,P_5)=\tilde O(u_2(n)n)$, which is slightly stronger than what is stated in Theorem \ref{main}. 

Instead of \eqref{2drichness} we need the following more general bound on the number of rich points.  
\begin{obs}[Richness bound]\label{richnessbound} Let $n^y$ be the maximum possible number of points that are $n^{\alpha}$-rich with respect to a set of $n^{x}$ points and some distance $\delta$. Then we have
\begin{equation}\label{rich1}
n^{y+\alpha}\leq u_2(n^x,n^y),
\end{equation}
or, equivalently
\begin{equation*}
n^{\alpha}\leq \frac{u_2(n^x,n^y)}{n^y}.
\end{equation*}
\end{obs}
The proof of \eqref{rich1} follows immediately from the definition of $n^{\alpha}$-richness and $u_2(n^x,n^y)$.

\medskip

Let $\Lambda:=\big\{\frac i{\log n}: i = 0,\ldots, \lceil\log n\rceil \big\}^4$. For any $\bm{\alpha}=(\alpha_2,\alpha_3,\alpha_4,\alpha_5) \in \Lambda$ let $Q_1^{\bm{\alpha}} = P_1$ and for $i=2,\ldots, 5$ define recursively $Q_i^{\bm\alpha}$ to be the set of those points in $P_i$ that are at least $n^{\alpha_i}$-rich but at most $2n^{\alpha_i}$-rich with respect to $Q_{i-1}$ and $\delta_i$.

It is not difficult to see that
\[\mathcal C_4(P_1,\ldots, P_5)= \bigcup_{\bm\alpha\in \Lambda} \mathcal{C}_4\left(Q_1^{\bm\alpha},\ldots,Q_5^{\bm\alpha}\right).\]
We have $|\Lambda| = \tilde O(1)$ and thus, in order to prove the theorem, it is sufficient to show that for every $\bm{\alpha}\in \Lambda$ we have 
\begin{equation*}
C_4\left(Q_1^{\bm{\alpha}},\dots,Q_5^{\bm{\alpha}}\right)= O\left(n \cdot u_2(n,n)\right).
\end{equation*}

From now on, fix $\bm{\alpha}=(\alpha_2,\dots,\alpha_5)$, and denote $Q_i=Q_i^{\bm\alpha}$. Choose $x_i\in [0,1]$ so that $|Q_i|=n^{x_i}$. Then we have \begin{equation}\label{eq51} C_4(Q_1,\dots,Q_5)= O\left(n^{x_5+\alpha_5+\alpha_4+\alpha_3+\alpha_2}\right).\end{equation}
Indeed, each chain $(p_1,\dots,p_5)$ with $p_i\in Q_i$ can be obtained in the following five steps.
\vspace{0.2cm}
\begin{itemize}
    \item \bf{Step 1: } \rm Pick $p_5\in Q_5$.
    \vspace{0.1cm}
    \item \bf{Step i ($2\le i\le 5$): } \rm Pick a point $p_{6-i}\in Q_{6-i}$ at distance $\delta_{6-i}$ from $p_{7-i}$.
\end{itemize}
\vspace{0.2cm}
In the first step we have $n^{x_5}$ choices, and for $i\geq 2$ in the $i$-th step we have at most $2n^{\alpha_{6-i}}$ choices. Further, by Observation~\ref{richnessbound}, for each $i\geq 2$ we have \begin{equation}\label{eq52}n^{\alpha_i}\leq \frac{u_2(n^{x_{i-1}},n^{x_i})}{n^{x_i}}.\end{equation}
Combining \eqref{eq51} and \eqref{eq52}, we obtain
\begin{equation}\label{eq53}
C_4(Q_1,\dots,Q_5)=O\left( u_2(n^{x_4},n^{x_5})\frac{u_2(n^{x_3},n^{x_4})}{n^{x_4}}\frac{u_2(n^{x_2},n^{x_3})}{n^{x_3}} \frac{u_2(n^{x_1},n^{x_2})}{n^{x_2}}\right).
\end{equation}
By \eqref{2drich} we have \[u_2(n^{x_{i-1}},n^{x_i})= O\left (\max \big\{n^{\frac{2}{3}(x_i+x_{i-1})},n^{x_i},n^{x_{i-1}}\big\}\right ).\]

Note that the maximum is attained on the second (third) term iff $x_{i-1}\le \frac {x_i}2$ ($x_i\le \frac{x_{i-1}}2$).
To bound $C_4(Q_1,\dots,Q_5)$
we consider several cases depending on which of these three terms the maximum above is attained on for different $i$. 

\bigskip

\bf{Case 1: }\rm For all $2\leq i \leq 5$ we have 
$u_2(n^{x_{i-1}},n^{x_i})= O\left( n^{\frac{2}{3}(x_i+x_{i-1})}\right)$.
Then 
\begin{equation*}
\frac{u_2(n^{x_4},n^{x_5})u_2(n^{x_3},n^{x_4})u_2(n^{x_2},n^{x_3})}{n^{x_2+x_3+x_4}}=O\left( n^{\frac{2}{3}x_5+\frac{1}{3}x_4+\frac{1}{3}x_3-\frac{1}{3}x_2}\right)
\end{equation*}
and
\begin{equation*}
\frac{u_2(n^{x_3},n^{x_4})u_2(n^{x_2},n^{x_3})u_2(n^{x_1},n^{x_2})}{n^{x_2+x_3+x_4}}=O\left( n^{-\frac{1}{3}x_4+\frac{1}{3}x_3+\frac{1}{3}x_2+\frac{2}{3}x_1}\right).
\end{equation*}

Substituting each of these two displayed formulas into \eqref{eq53} and taking their product, we obtain
\[C_4(Q_1,\dots,Q_5)^2 = O\left( u_2(n^{x_1},n^{x_2})u_2(n^{x_4},n^{x_5})\cdot n^{\frac{2}{3}x_1+\frac{2}{3}x_3+\frac{2}{3}x_5}\right)= O\left( u_2(n,n)^2\cdot n^{2}\right),\]
which concludes the proof in this case.

\bigskip

\bf{Case 2: } \rm 
There is an $2\leq i \leq 5$ such that \begin{equation}\label{eq55} \min \{x_{i-1},x_i\}\le \frac 12\max\{x_{i-1},x_i\}\ \ \ \text{and thus}\ \ \ \ u_2(n^{x_{i-1}},n^{x_i})= O\left (\max \{n^{x_{i-1}},n^{x_{i}}\}\right).\end{equation} 
We distinguish three cases based on for which $i$ holds.

\bigskip 

{\bf Case 2.1:} \eqref{eq55} holds for $i=2$ or  $5$. In particular, this implies that $u_2(n^{x_{1}},n^{x_2})= O(n)$ or $u_2(n^{x_{4}},n^{x_5})= O(n)$. The following lemma finishes the proof in this case.

\begin{lem}\label{smalln}
Let $R_1,\ldots, R_5\subseteq \mathbb{R}^2$ such that $|R_i|\leq n$ for every $i\in [5]$. If $u_2(R_1,R_2)= O(n)$ or $u_2(R_4,R_5)= O(n)$ holds, then
$C_4(R_1,\dots,R_5)= O\left(n\cdot u_2(n,n)\right)$.
\end{lem}

\begin{proof} We have \[C_4(R_1,\dots,R_5)\leq 2u_2(R_1,R_2)u_2(R_4,R_5)=O\left(n\cdot u_2(n,n)\right).\] Indeed,  every $4$-tuple $(r_1,r_2,r_4,r_5)$ with $r_i\in R_i$ can be extended in at most two different ways to a $4$-chain $(r_1,\dots,r_5)\in R_1\times\dots\times R_5$. At the same time, the number of $4$-tuples with $\|r_1-r_2\|=\delta_1$, $\|r_4-r_5\|=\delta_4$ is at most $u_2(R_1,R_2)u_2(R_4,R_5)$.
\end{proof}

\bigskip

\bf{Case 2.2: }\rm \eqref{eq55} holds for $i=4.$ Note that if  $x_4\leq \frac{x_3}{2}\le \frac 12$, then  $u_2(n^{x_5},n^{x_4})=O(n)$, and we can apply Lemma~\ref{smalln} to conclude the proof in this case. Thus we may assume that $x_3\le \frac{x_4}2$, and hence $u_2(n^{x_4},n^{x_3})=O(n^{x_4})$.
This means that $n^{\alpha_4}=O(1)$ by Observation \ref{richnessbound}. Thus to finish the proof of this case, it is sufficient to prove the following claim.
\begin{cla}
Let $R_1,\dots,R_5\subseteq \mathbb{R}^2$ such that $|R_i|\leq n$ for all $i\in [5]$ and every point of $R_4$ is $O(1)$ rich with respect to $R_3$ and $\delta_3$. Then $C_4(R_1,\dots,R_5) = O\left(n\cdot u_2(n,n)\right)$.
\end{cla}
\begin{proof} Every $4$-chain $(r_1,\dots,r_5)$ can be obtained in the following steps.
\vspace{0.2cm}
\begin{itemize}
    \item Pick a pair $(r_4,r_5)\in R_4\times R_5$ with $\|r_4-r_5\|=\delta_4$.
    \vspace{0.1cm}
    \item Choose $r_3\in R_3$ at distance $\delta_3$ from $r_4$.
    \vspace{0.1cm}
    \item Pick a point $r_1\in R_1$.
    \vspace{0.1cm}
    \item Extend $(r_1,r_3,r_4,r_5)$ to a $4$-chain.
\end{itemize}

In the first step, we have at most $u_2(n,n)$ choices, in the third at most $n$ choices, and in the other two steps at most $O(1)$. 
\end{proof}

{\bf Case 2.3 } \eqref{eq55} holds for $i=3$ {\it only}. Arguing as in Case 2.2, we may assume that $u_2(n^{x_3},n^{x_2})=O(n^{x_2})$. Then we have 
\begin{multline*}
C_4(Q_1,\dots,Q_5)= O\left( u_2(n^{x_4},n^{x_5})\frac{u_2(n^{x_3},n^{x_4})}{n^{x_4}}\frac{u_2(n^{x_2},n^{x_3})}{n^{x_3}} \frac{u_2(n^{x_1},n^{x_2})}{n^{x_2}}\right)\\[10pt] 
= O\left(u_2(n^{x_1},n^{x_2})\cdot n^{\frac{2}{3}(x_4+x_5)+\frac{2}{3}(x_3+x_4)-x_4-x_3}\right)= O\left(u_2(n,n)\cdot n\right),
\end{multline*}
which finishes the proof.

\subsection{Upper bound for the \texorpdfstring{$\bm{k\equiv 1}$}{Lg} (mod \texorpdfstring{$\bm{3}$)}{Lg} case}\label{sec34}
We will prove the upper bound in Theorem \ref{main} for $k\equiv 1$ by induction. The $k=1$ case follows from the definition of $u_2(n,n)$, thus we may assume that $k\geq 4$. For the rest of the  section fix $\eps'>0$, and sets $P_1,\ldots, P_{k+1}\subseteq \mathbb{R}^2$ of size $n$, further let $\varepsilon=\frac{\varepsilon'}{4k}$. We are going to show that $C_k(P_1,\dots,P_{k+1})=O(n^{(k-1)/3+\varepsilon'}u_2(n))$. 

The first step of the proof is to find a certain covering of $P_1\times \dots \times P_{k+1}$, which resembles the one used for the $k=4$ case, although is more elaborate.\footnote{This covering brings in the $\varepsilon$-error term in the exponent, that we could avoid in the $k=4$ case.} (The goal of this covering is to make the corresponding graph between each of the two consecutive parts `regular in both directions' in a certain sense.)

Let \[\Lambda=\Big\{i\eps: i=0,\ldots, \Big\lfloor\frac 1\eps\Big\rfloor\Big\}^{k+1}.\] We cover the product ${\bf P}=P_1\times\dots\times P_{k+1}$ by fine-grained classes $P_1^{\bm{\gamma}}\times \ldots\times P_{k+1}^{\bm{\gamma}}$ encoded by the sequence $\bm{\gamma} = (\bm{\gamma^1},\bm{\gamma^2},\ldots )$ of length at most $(k+1)\eps^{-1}+1$ with $\bm{\gamma^j}\in \Lambda$ for each $j=1,2,\dots$. One property that we shall have is 
\begin{equation*} P_1\times\dots \times P_{k+1} = \bigcup_{\bm{\gamma}}P_1^{\bm\gamma}\times \ldots\times P_{k+1}^{\bm\gamma}.
\end{equation*}

To find the covering, first we define a function $D$ that receives a parity digit $j\in \{0,1\}$, a product set ${\bf R}:= R_1\times\ldots \times R_{k+1}$ and an $\bm{\alpha}\in \Lambda$, and outputs a product set $D(j,\bm{R},\bm{\alpha})={\bf R(\bm \alpha)}= R_1(\bm\alpha)\times\ldots \times R_{k+1}(\bm\alpha)$.
\medskip

{\bf Definition of $\mathbf{D}$}
\begin{itemize}
    \item If $j=1$ then let $R_1(\bm{\alpha}):=R_1$ and for $i=2,\ldots, k+1$ define $R_i(\bm{\alpha})$ iteratively to be the set of points in $R_i$ that are at least $n^{\alpha_i}$, but at most $n^{\alpha_i+\eps}$-rich with respect to $R_{i-1}(\bm{\alpha})$ and $\delta_{i-1}.$
     \item If $j=0$ then apply the same procedure, but in reverse order. That is, let $R_{k+1}(\bm{\alpha})=R_{k+1}$ and for $i=k,k-1,\dots,1$ define $R_{i}(\bm{\alpha})$ iteratively to be the set of points in $R_i$ that are at least $n^{\alpha_i}$ but at most $n^{\alpha_i+\eps}$-rich with respect to $R_{i+1}(\bm{\alpha})$ and $\delta_i$. 
\end{itemize}
\vspace{0.2cm}
Note that \begin{equation}\label{equnion}{\bf R}= \bigcup_{\bm{\alpha}\in \Lambda} {\bf R}(\bm{\alpha}).\end{equation}
For a sequence $\bm{\gamma} = (\bm{\gamma^1},\bm{\gamma^2}, \ldots)$ with $\bm{\gamma^j}\in \Lambda,$ we define ${\bf P}^{\bm\gamma}$ recursively as follows. Let ${\bf P}^{\emptyset} :={\bf P}$, and for each $j\geq 1$ let
\[{\bf{P}}^{(\bm{\gamma^1},\ldots,\bm{\gamma^{j}})}=D(j \text{ (mod }2),{\bf{P}}^{(\bm{\gamma^1},\ldots,\bm{\gamma^{j-1}})},\bm{\gamma^j}).\]
We say that a sequence $\bm{\gamma}$ is {\it stable at $j$} if \begin{equation*}\label{eqdecomp} \big|{\bf P}^{(\bm{\gamma^1},\ldots,\bm{\gamma^{j}})}\big|\ge \big|{\bf P}^{(\bm{\gamma^1},\ldots,\bm{\gamma^{j-1}})}\big|\cdot n^{-\eps}.\end{equation*}
Otherwise $\bm{\gamma}$ is \emph{unstable at $j$}.

\begin{defi}\label{defups} Let $\Upsilon$ be the set of those sequences $\bm{\gamma}$ that are stable at their last coordinate, but are not stable for any previous coordinate, and for which ${\bf P}^{\bm\gamma}$ is non-empty.
\end{defi} 

The set $\Upsilon$ has several useful properties, some of which are summarised in the following lemma. 

\begin{lem}\label{decompprop} 
1. Any $\bm{\gamma}\in \Upsilon$ has length at most $(k+1) \eps^{-1}+1$. \\ \vspace{0.1cm}
2. $|\Upsilon| = O_{\eps}(1).$ \\ \vspace{0.1cm}
3. ${\bf P} = \bigcup_{\bm{\gamma}\in \Upsilon}{\bf P}^{\bm\gamma}.$\\
\end{lem}

\begin{proof}
\begin{enumerate}
  \item If $\bm{\gamma}$ is unstable at $j$ then \[|{\bf P}^{(\bm\gamma^1,\ldots,\bm\gamma^{j})}|\le |{\bf P}^{(\bm\gamma^1,\ldots,\bm\gamma^{j-1})}|\cdot n^{-\eps}.\]
  Since $|{\bf P}| = n^{k+1}$ and $|{\bf P}^{\bm\gamma}|\ge 1,$ we conclude that $\bm{\bm\gamma}$ is unstable at at most $(k+1) \eps^{-1}$ indices $j$. \vspace{0.1cm}
  
  \item It follows from part 1 by counting all possible sequences of length at most $(k+1) \eps^{-1}+1$ of elements from the set $\Lambda$. (Note that $|\Lambda| = O_{\eps}(1).$) \vspace{0.1cm}

  \item For a nonnegative integer $j$ let $\Lambda^{\le j}$ be the set of all sequences of length at most $j$ of elements from $\Lambda$. Let 
  \[\Upsilon_j:= \left(\Upsilon\cap \Lambda^{\le j}\right)\cup\Psi_j,  \textrm{ where } \ \Psi_j:=\big\{\bm{\gamma}\in \Lambda^j: \bm{\gamma} \text{ is not stable for any }\ell \le j\big\}.\]
By part 1 of the lemma, $\Upsilon_j = \Upsilon$ for $j>(k+1)\eps^{-1}.$
  We prove by induction on $j$ that ${\bf P}= \bigcup_{\bm{\gamma}\in \Upsilon_j}{\bf P}^{\bm\gamma}$.

  $\Upsilon_0$ consists of an empty sequence, thus the statement is clear for $j=0$. Next, assume that the statement holds for $j$. 
  We have 
  \[{\bf P}=\bigcup_{\bm{\gamma}\in \Upsilon_j}{\bf P}^{\bm\gamma}=\bigcup_{\bm{\gamma}\in \Lambda^{\leq j}}{\bf P}^{\bm\gamma}\cup\bigcup_{\bm{\gamma}\in \Psi_j}{\bf P}^{\bm\gamma}.\]
  
  By $\eqref{equnion}$ we have that ${\bf P}^{\bm\gamma} = \bigcup_{\bm{\gamma'}}{\bf P}^{\bm\gamma'}$ holds for any $\bm\gamma\in \Psi_j$, where the union is taken over the sequences from $\Lambda^{j+1}$ that coincide with $\bm{\gamma}$ on the first $j$ entries. This, together with  $\bm\gamma'\in \left(\Upsilon\cap \Lambda^{j+1}\right)\cup \Psi_{j+1}$ when ${\bf P}^{\bm\gamma'}$ is nonempty finishes the proof.
 \end{enumerate}
\end{proof}

Parts 2 and 3 of Lemma~\ref{decompprop} imply that in order to complete the proof of the \mbox{$k\equiv 1$ (mod $3$)} case, it is sufficient to show that for any $\bm{\gamma}\in \Upsilon$ we have
\begin{equation}\label{eq61}
C_k(P_1^{\bm\gamma},\ldots, P_{k+1}^{\bm\gamma})= O\left(u_2(n)\cdot n^{\frac{k-1}3+4k\eps}\right).
\end{equation}

From now on fix $\bm{\gamma}\in \Upsilon$. For each $i=1,\ldots, k+1$ let $R_i:=P_i^{\bm\gamma}$ and $Q_i:= P_i^{\bm\gamma'}$, where $\bm{\gamma'}$ is obtained from $\bm{\gamma}$ by removing the last element of the sequence.
Without loss of generality, assume that the length $\ell$ of $\bm{\gamma}$ is even. For each $i=1,\ldots, k+1,$ choose $x_i,y_i$ such that 
\[|Q_i| = n^{x_i}, \ \ \ \ |R_i| = n^{y_i}.\]

Let $\alpha_i:= \bm\gamma^{\ell-1}_i$ and $\beta_i:= \bm\gamma^{\ell}_i$. By the definition of $\bf{P}^{\bm{\gamma}}$ we have that each point in $Q_i$ is at least $n^{\alpha_i}$-rich but at most $n^{\alpha_i+\varepsilon}$-rich with respect to  $Q_{i-1}$ and $\delta_{i-1}$, and each point in $R_i$ is at least $n^{\beta_i}$-rich but at most $n^{\beta_i+\varepsilon}$-rich with respect to  $R_{i+1}$ and $\delta_i$.

By Observation~\ref{richnessbound}, we have \begin{equation}\label{eqkey} 
n^{\alpha_i}\leq \frac{u_2(n^{x_{i-1}},n^{x_i})}{n^{x_i}}\ \  \ \ \text{and} \ \ \ \ n^{\beta_i}\leq \frac{u_2(n^{y_{i}},n^{y_{i+1}})}{n^{y_i}}\le  \frac{u_2(n^{x_{i}},n^{x_{i+1}})}{n^{x_i-\eps}}.
\end{equation}
The last inequality follows from two facts: first $u_2(n^{y_{i}},n^{y_{i+1}})\le u_2(n^{x_{i}},n^{x_{i+1}})$ and, second, since $\bm\gamma$ is stable at its last coordinate\footnote{This is the only place where we use the stability of $\bm\gamma$ directly.}, we have $n^{y_i} = |R_i|\ge |Q_i|\cdot n^{-\eps} = n^{x_i-\eps}.$

In the same fashion as in the beginning of Section~\ref{sec33}, we can show that
\begin{align*}
C_k(R_1,\dots,R_{k+1})\leq& n^{y_{1}}n^{\beta_{1}+\dots+\beta_k+k\varepsilon}, \textrm{ and }\\[5pt]
C_k(R_1,\dots,R_{k+1})\leq C_k(Q_1,\dots,Q_{k+1})\leq& n^{x_{k+1}}n^{\alpha_{k+1}+\alpha_{k}+\dots+\alpha_2+k\varepsilon}.
\end{align*}
Combining the first of these displayed inequalities with \eqref{eqkey}, we have 
\begin{equation*}
C_{k}(R_1,\dots,R_{k+1})
\leq  u_2(n^{x_1},n^{x_{2}})\prod_{2\leq i \leq k}\frac{u_2\left (n^{x_{i}},n^{x_{i+1}}\right)}{n^{x_i}}n^{2k\varepsilon} .
\end{equation*}
Recall that 
\begin{equation}\label{eq66}
u_2(n^{x_{i}},n^{x_{i+1}})=O\left (\max \{n^{\frac{2}{3}(x_i+x_{i+1})},n^{x_i},n^{x_{i+1}}\}\right ).
\end{equation}
To bound $C_k(R_1,\dots,R_{k+1})$,
we consider several cases based on which of these three terms can be used to bound $u_2(n^{x_{i}},n^{x_{i+1}})$ for different values of $i$.

\bigskip

{\bf Case 1: }
 Either $u_2(n^{x_1},n^{x_2})= O(n)$ or $u_2(n^{x_k},n^{x_{k+1}})= O(n)$ holds.
As in the proof of Lemma~\ref{smalln}, we have 
\begin{multline*}
    C_k(R_1,\dots,R_{k+1})\\[5pt] \leq  
    \min\big\{ 2u_2(n^{y_1},n^{y_{2}})C_{k-3}(R_4,\dots,R_{k+1}),2u_2(n^{y_k},n^{y_{k+1}})C_{k-3}(R_1,\dots,R_{k-2})\big \}.
\end{multline*}
By induction we obtain $C_{k-3}(R_4,\dots,R_{k+1}),C_{k-3}(R_1,\dots,R_{k-2})= O\left(n^{\frac{k-4}{3}+\varepsilon}\cdot u_2(n)\right )$. Together with the assumption of Case 1, and the fact that $u_2(n^{y_1},n^{y_{2}})\leq u_2(n^{x_1},n^{x_{2}})$ and \mbox{$u_2(n^{y_k},n^{y_{k+1}})\leq u_2(n^{x_k},n^{x_{k+1}})$}, this implies \eqref{eq61} and finishes the proof.\\

{\bf Case 2: } 
For some $i=1,\ldots, (k-1)/3,$ one of the following holds: 
\vspace{0.3cm}
\begin{itemize}
    \item $u_2(n^{x_{3i+1}},n^{x_{3i+2}})= O(\max\{n^{x_{3i+1}},n^{x_{3i+2}}\})$; 
    \vspace{0.1cm}
    \item $u_2(n^{x_{3i-1}},n^{x_{3i}})= O(n^{x_{3i-1}})$;
    \vspace{0.1cm}
    \item $u_2(n^{x_{3i}},n^{x_{3i+1}})= O(n^{x_{3i+1}})$.
\end{itemize}

We will show how to conclude in the first case. The other cases are very similar and we omit the details of their proofs. 
If $u_2(n^{x_{3i+1}},n^{x_{3i+2}})= O(n^{x_{3i+2}})$ then $n^{\alpha_{3i+2}}=O(1)$ by \eqref{eqkey}. Every chain  $(r_1,\dots,r_{k+1})\in \mathcal C_k(Q_1,\ldots, Q_{k+1})$ 
can be obtained as follows. 
\vspace{0.3cm}
\begin{enumerate}
    \item Pick a $(3i-2)$-chain $(r_1,\dots,r_{3i-1})$ with $r_j\in Q_j$ for every $j$.
    \vspace{0.2cm}
    \item Pick a $(k-3i-1)$-chain $(r_{3i+2},r_{3i+3},\dots,r_{k+1})$ with $r_j\in Q_j$ for every $j$.
    \vspace{0.2cm}
    \item Extend $(r_{3i+2},r_{3i+3},\dots,r_{k+1})$ to a $(k-3i-2)$ chain $(r_{3i+1},r_{3i+2},\dots,r_{k+1})$.
    \vspace{0.2cm}
    \item Connect $(r_1,\dots,r_{3i-1})$ and $(r_{3i+1},r_{3i+2},\dots,r_{k+1})$ to obtain a $k$-chain.
\end{enumerate}

In the first step, we have $O\left(n^{\frac{3i-3}3+\varepsilon}\cdot u_2(n)\right)$ choices by induction on $k$. In the second step, we have $\tilde O\left(n^{\frac{k-3i+2}{3}}\right)$ choices by the $k\equiv 0$ (mod $3$) case of Theorem~\ref{main}. In the third step, we have at most $n^{\alpha_{3i+2}+\varepsilon}= O(n^{\varepsilon})$ choices. Finally, in the fourth step we have at most $2$ choices. Thus the number of $k$-chains is at most \[O\left(n^{\frac{3i-3}{3}+\varepsilon}\cdot u_2(n)\right)\cdot \tilde O\left (n^{\frac{k-3i+2}{3}}\right )\cdot O\left(n^{\varepsilon}\right)\cdot 2=O\left(n^{\frac{k-1}{3}+3\varepsilon}\cdot u_2(n)\right),\]
finishing the proof of the first case.

\smallskip
If $u_2(n^{x_{3i+1}},n^{x_{3i+2}})= O(n^{x_{3i+1}})$ then  $n^{\beta_{3i+1}}= O( n^{\varepsilon})$ by \eqref{eqkey}.\footnote{This is the key application of \eqref{eqkey}, and the reason why we needed a decomposition with regularity in both directions between the consecutive parts.}
We proceed similarly in this case, but we count the $k$-chains now in $R_1\times\ldots\times R_{k+1}$ instead in $Q_1\times \ldots \times Q_{k+1}$ (and get an extra factor of $n^{\eps}$ in the bound). In all cases, we obtain \eqref{eq61}.

\bigskip

{\bf Case 3: } Neither the assumptions of Case 1 nor that of Case 2 hold.  We define four sets $S'$, $S'_+$, $S'_{++}$, and $S'_-$ of indices in $\{2,\ldots, k\}$ as follows. Let

\begin{flalign*}
S'& := \setbuilder{i}{u_2(n^{x_i},n^{x_{i-1}})= O(n^{\frac{2}{3}(x_i+x_{i-1})}) \textrm{ and } u_2(n^{x_{i+1}},n^{x_{i}})= O(n^{\frac{2}{3}(x_{i+1}+x_{i})}) }, \\[7pt]
S'_+& := \Big\{i : u_2(n^{x_i},n^{x_{i-1}})= O(n^{\frac{2}{3}(x_i+x_{i-1})}) \textrm{ and } u_2(n^{x_{i+1}},n^{x_{i}})= O(n^{x_{i}}), \textrm{ or } \\
 & \phantomrel{=}{}  \phantomrel{=}{} \phantomrel{=}{} \phantomrel{=}{} u_2(n^{x_i},n^{x_{i-1}})= O(n^{x_i}) \textrm{ and } u_2(n^{x_{i+1}},n^{x_{i}})= O(n^{\frac{2}{3}(x_{i+1}+x_{i})}) 
\Big \}, \\[7pt]
 S'_{++}& :=\Big\{i : u_2(n^{x_i},n^{x_{i-1}})= O(n^{x_i}) \textrm{ and } u_2(n^{x_{i+1}},n^{x_{i}})= O(n^{x_{i}})
\Big \}, \textrm{ and} \\[7pt]
S'_-& :=\Big\{i : u_2(n^{x_i},n^{x_{i-1}})= O(n^{\frac{2}{3}(x_i+x_{i-1})}) \textrm{ and } u_2(n^{x_{i+1}},n^{x_{i}})= O(n^{x_{i+1}}), \textrm{ or } \\
& \phantomrel{=}{}  \phantomrel{=}{} \phantomrel{=}{} \phantomrel{=}{} u_2(n^{x_i},n^{x_{i-1}})= O(n^{x_{i-1}}) \textrm{ and } u_2(n^{x_{i+1}},n^{x_{i}})= O(n^{\frac{2}{3}(x_{i+1}+x_{i})}) 
\Big \}.
\end{flalign*}

Since the conditions of Case 2 are not satisfied, we have \[\{2,\dots,k\}\subseteq S'\cup S'_+\cup S'_{++}\cup S'_-.\]
Indeed, for each $i\in\{2,\ldots, k\},$ there are $9$ possible pairs of maxima in \eqref{eq66} with $i,i+1.$ The four sets above encompass $6$ possibilities. In total, there are $4$ possible pairs of maxima with only the two last terms from \eqref{eq66} used. For $i \equiv 1,2$ (mod $3$), any of those $4$ are excluded due to the first condition in Case 2 (in fact, then $i\in S'\cup S'_-$). If $i \equiv 0$ (mod $3$), then the second and the third condition in Case 2 rule out all possibilities but the one defining $S'_{++}$.

From these it directly follows that if $i\in S'_{++}$, then $i-1,i+1\in S'_-$, while if $i\in S'_+$ then one of $i-1,i+1$ is in $S'_-$. (Recall that $i\in S'_+\cup S'_{++}$ only if $i\equiv 0$ (mod $3$).) These together imply 
\begin{equation}\label{eq67}
    |S'_+|+2|S'_{++}|\le |S'_-|.
\end{equation}

We partition $\{2,\dots,k\}$ using these sets as follows: let $S_- = S'_-, S = S'\setminus S'_-, S_+ = S'_+\setminus (S'_-\cup S')$ and $S_{++}=\{2,\dots,k\}\setminus S'_-\cup S'\cup S'_{+}$. Note that the analogue of \eqref{eq67} holds for the new sets. That is, we have
\begin{equation*}\label{eq68}
    |S_+|+2|S_{++}|\le |S_-|.
\end{equation*}

Recall that 
\begin{equation}\label{24}
C_{k}(R_1,\dots,R_{k+1})
\leq  u_2(n^{x_1},n^{x_{2}})\prod_{2\leq i \leq k}\frac{u_2\left (n^{x_{i}},n^{x_{i+1}}\right)}{n^{x_i}}n^{2k\varepsilon}. 
\end{equation}
Since the assumptions of Case 1 and 2 do not hold, we have $2,k\in S$. Indeed, $2,k\ne 0$ \mbox{(mod $3$)} and thus $2,k\notin S_+,S_{++}$. Further, if say $k\in S_-=S'_-$ then by the definition of $S'_-$ we either have $u_2(n^{x_{k+1}},n^{x_k}) = O(n)$, or $u_2(n^{x_k},n^{x_{k-1}})=O(n^{x_{k-1}})$. The first case cannot hold since the assumption of Case 1 does not hold. Further, the second case cannot hold either, since it would imply $x_k\le \frac {x_{k-1}}2\le \frac 12$, meaning $u_2(n^{x_{k+1}},n^{x_k}) = O(n)$.
Using $2,k\in S$ and expanding \eqref{24}, we obtain
{\small \begin{equation}\label{first}
C_{k}(R_1,\dots,R_{k+1})\leq  
 n^{2k\varepsilon}u_2(n^{x_1},n^{x_2})n^{-\frac{1}{3}x_2}n^{\frac{2}{3}x_{k+1}}\prod_{\substack{i\in S,\\i\neq 2}}n^{\frac{1}{3}x_i}\prod_{i\in S_+}n^{\frac{2}{3}x_i}\prod_{i\in S_{++}}n^{x_i}\prod_{i\in S_-}n^{-\frac{1}{3}x_i},
\end{equation}}
and
{\small \begin{equation}\label{second}
 C_{k}(R_1,\dots,R_{k+1})\leq 
 n^{2k\varepsilon}u_2(n^{x_k},n^{x_{k+1}})n^{-\frac{1}{3}x_k}n^{\frac{2}{3}x_1}\prod_{\substack{i\in S,\\ i\neq k}}n^{\frac{1}{3}x_i}\prod_{i\in S_+}n^{\frac{2}{3}x_i}\prod_{i\in S_{++}}n^{x_i}\prod_{i\in S_-}n^{-\frac{1}{3}x_i}.
\end{equation}}

Taking the product of \eqref{first} and \eqref{second} we obtain 

\begin{multline*}
 C_{k}(R_1,\dots,R_{k+1})^2 \leq \\  n^{4k\varepsilon}\cdot u_2(n^{x_1},n^{x_2})u_2(n^{x_k},n^{x_{k+1}})n^{\frac{2}{3}(x_1+x_{k+1})}\left (\prod_{ \substack{i\in S,\\ i\neq 2,k}} n^{\frac{1}{3}x_i}\prod_{i\in S_+}n^{\frac{2}{3}x_i}\prod_{i\in S_{++}}n^{x_i}\prod_{i\in S_-}n^{-\frac{1}{3}x_i} \right)^2\\[10pt] \leq 
 n^{4k\varepsilon}\cdot u_2(n,n)^2\cdot n^{2\left(\frac{2}{3}+ \frac{1}{3}|S\setminus\{2,k\}|+\frac{2}{3}|S_+|+|S_{++}|\right ) }= u_2(n,n)^2\cdot n^{\frac{2(k-1)}{3}+4k\eps}.
\end{multline*}
The last equality follows from $|S_+|+2|S_{++}|\le |S_-|$, which is equivalent to $\frac 23|S_+|+|S_{++}|\le \frac 13(|S_+|+|S_{++}|+|S_-|)$, and from the fact that $S$, $S_+$,$S_{++}$, and $S_-$ partition $\{2,\dots,k\}$. This finishes the proof.

\section{Bounds in \texorpdfstring{\bm{$\mathbb{R}^3$}}{Lg}}

Similarly as in the planar case, for a fixed $\bm\delta=(\delta_1,\dots,\delta_k)$ and $P_1\dots,P_{k+1}\subseteq \mathbb{R}^3$ we denote by $\mathcal{C}_k^3(P_1,\dots,P_k)$ the family of $(k+1)$-tuples $(p_1,\dots,p_{k+1})$ with $p_i\in P_i$ for all $i\in[k+1]$ and with $\|p_i-p_{i+1}\|=\delta_i$ for all $i\in[k]$. Let $C_k^3(P_1,\dots,P_{k+1})=|\mathcal{C}_k^{3}(P_1,\dots,P_{k+1})|$ and 
\[C^3_k(n_1,\dots,n_{k+1})=\max C_k^{3}(P_1,\dots,P_{k+1}),\]
where the maximum is taken over all choices of $P_1,\ldots, P_{k+1}$ subject to $|P_i|\le n_i$ for all $i\in [k+1]$.

Similarly to the planar case it follows that $C^3_k(n)\leq C^3_k(n,\dots,n)\leq C^3_k\left((k+1)n\right)$. Since we are only interested in the order of magnitude of $C^3_k(n)$ for fixed $k$, sometimes we are going to work with $C^3_k(n,\dots,n)$ instead of $C^3_k(n)$.

\subsection{Lower bounds}

For completeness, we recall the constructions from \cite{Shef} for even $k\geq 2$. Let $\bm{\delta}=(\delta_1,\dots,\delta_k)$ be any given sequence. For every even $2\leq i\leq k$, let $P_i=\{p_i\}$ be a single point such that the sphere of radius $\delta_i$ centred at $p_i$ and the sphere of radius $\delta_{i+1}$ centred at $p_{i+2}$ intersect in a circle. Further, let $P_1$ be a set of $n$ points contained in the sphere of radius $\delta_1$ centred at and $p_2$, and $P_{k+1}$ be a set of $n$ points contained in the sphere of radius $\delta_k$ centred at and $p_2$.
Finally, for every odd $3\leq i \leq k-1$, let $P_i$ be a set of $n$ points contained in the intersection of the sphere of radius $\delta_{i-1}$ centred at $p_{i-1}$ and of the sphere of radius $\delta_{i}$ centred at $p_{i+1}$. Then $P_1\times\dots\times P_{k+1}$ contains $n^{\frac{k}{2}+1}$ many $k$-chains, since every element of $P_1\times\dots\times P_{k+1}$ is a $k$-chain, and $|P_1\times\dots\times P_{k+1}|=n^{\frac{k}{2}+1}$.

Next, we prove the lower bounds for odd $k\geq 3$ and $\bm\delta=(1,\dots,1)$ given in Proposition \ref{oddk}.

\begin{proof}[Proof of Proposition \ref{oddk}]

First we show that $C_k^3(n)=\Omega\left (\frac{u_3(n)^k}{n^{k-1}}\right )$. Take a set $P'\subset \R^3$ of size $n$ that contains $u_3(n)$ point pairs at unit distance apart. It is a standard exercise in graph theory to show that since $u_3(n)$ is superlinear, there is $P\subset P'$ such that $\frac{n}{2}\leq |P|\leq n$ and for every $p\in P$ there are at least $\frac{u_3(n)}{4n}$ points $p'\in P$ at distance $1$ from $p$. Then $P$ contains $\Omega\left (\frac{u_3(n)^k}{n^{k-1}}\right )$ many $k$-chains with $\bm\delta=(1,\dots,1)$.

To prove $C_k^3(n)=\Omega\left (us_3(n)n^{(k-1)/2}\right)$, we modify and extend the construction used for $k-1$ as follows. Let $P_1,\dots,P_{k-1}$ be as in the construction for $(k-1)$-chains with $\bm\delta=(1,\dots,1)$ (from the even case). Further, let $P_{k}$ be a set of $n$ points on the unit sphere around $p_{k-1}$, and $P_{k+1}$ be a set of $n$ points such that $u_3(P_k,P_{k+1})=us_3(n)$. Since every $(p_1,\dots,p_{k+1})\in P_1\times \dots \times P_{k+1}$ with $\|p_k-p_{k+1}\|=1$ is a $k$-chain, we obtain that $P_1\times\dots\times P_{k+1}$ contains $\Omega\left (us_3(n)n^{(k-1)/2}\right)$ many $k$-chains with $\bm\delta=(1,\dots,1)$.
\end{proof}

\subsection{Upper bound}

We again fix  $\bm\delta=(\delta_1,\dots,\delta_k)$ throughout the section.
The following result with $x=1$ implies the upper bound in Theorem~\ref{3d}.

\begin{thm}\label{3dbound} For any fixed integer $k\geq 0$ and $x\in [0,1]$, we have 
\begin{equation*}
C^3_k(n^{x},n,\ldots,n)=\tilde O\left (
  n^{\frac{k+1+x}{2}}\right ).
\end{equation*}
\end{thm}

\begin{proof}
The proof is by induction on $k$. For $k=0$ the bound is trivial, and for $k=1$ it follows from \eqref{eq3d}.

For $k\geq 2$ let $P_1,\dots,P_{k+1}\subseteq \mathbb{R}^3$ be sets of points satisfying  $|P_1|=n^x$, and $|P_i|=n$ for $2\leq n \leq k+1$. Denote by $P_2^{\alpha}\subseteq P_2$ the set of those points in $P_2$ that are at least $n^{\alpha}$-rich but at most $2n^{\alpha}$-rich with respect to $P_1$ and $\delta_1$. 

It is not hard to see that
\[\mathcal{C}^3_k(P_1,P_2\dots,P_{k+1})\subseteq \bigcup_{\alpha\in \Lambda} \mathcal{C}^3_k(P_1,P_2^{\alpha},P_3,\ldots,P_{k+1}),\]
where $\Lambda:=  \{\frac{i}{\log n}: i= 0,1,\ldots, \lfloor \log n\rfloor\}$. Since $|\Lambda| = \tilde O(1)$, it is sufficient to prove that, for every $\alpha\in \Lambda,$ we have 
\begin{equation*}\label{alpha3}
C^3_k(P_1,P_2^{\alpha},P_3,\ldots, P_{k+1})= \tilde O\left ( n^{\frac{k+1+x}2}\right ).
\end{equation*}

Assume that $|P^{\alpha}_2| = n^y.$ The number of $(k-1)$-chains in $P_2^{\alpha}\times P_3\times \dots\times P_{k+1}$ is at most $C^3_{k-1}(n^{y},n,\dots,n)$, and each of them may be extended in $2n^{\alpha}$ ways.  By induction, we get
\[C^3_k(P_1,P_2^{\alpha},P_3,\ldots, P_{k+1})= \tilde O\left(n^{\alpha}\cdot n^{\frac{k+y}{2}}\right),\]
and we are done as long as \begin{equation}\label{eqtocheck}2\alpha+k+y\le k+1+x.\end{equation}
To show this, we need to consider several cases depending on the value of $\alpha.$
Note that $\alpha\le x$.
\vspace{0.1cm}
\begin{itemize}
\item If  $\alpha\geq \frac{2x}{3}$, then by \eqref{richness} we have $y\leq x-\alpha$, and the LHS of \eqref{eqtocheck} is at most $\alpha+k+x\le 1+k+x$.
\vspace{0.1cm}
\item If $\frac x2 \le \alpha\le \frac{2x}{3}$ then by \eqref{richness} we have $y\leq 3x-4\alpha.$ The LHS of \eqref{eqtocheck} is at most $k+3x-2\alpha\le k+2x\le k+1+x$.
\vspace{0.1cm}
\item If $\alpha\le \frac x2$ then we use a trivial bound  $y\leq 1$. The LHS of \eqref{eqtocheck} is at most $2\alpha+k+1\le x+k+1.$
\end{itemize}
\end{proof}

\section{Concluding remarks}

If $\bf{\delta}=(1,\dots,1)$, then the results of this paper are about the finding the maximum number of $k$-long paths that can be determined by a unit-distance graph on $n$ vertices. As a generalisation, one can study the maximum number of isomorphic copies of a given tree in unit distance graphs. More generally, we propose the following problem. Let $T=(V,E)$ be a tree with $V=\{v_1,\dots,v_{k+1}\}$ and  $E=\{(v_{i_1},v_{j_1}),\dots,(v_{i_k},v_{j_k})\}$. For a fixed sequence $\bm{\delta}=\{\delta_1,\dots,\delta_k\},$ a $(k+1)$-tuple of distinct points $(p_1,\dots,p_{k+1})$ in $\mathbb{R}^d$ is a \emph{$T$-tree}, if $\|p_{i_\ell}-p_{j_\ell}\|=\delta_{\ell}$ for every $\ell=1,\ldots, k$. What is the maximum possible number $C_T^d(n)$ of $T$-trees in a set of $n$ points in $\R^d$? For $d=2$ we write $C_T^2(n)=C_T(n)$.

Note that, as in the case of chains, $C_T^d(n) = \Omega\big(n^{|V(T)|}\big)=\Omega(n^{k+1})$ for $d\ge 4$.
However, for $d=2$, $C_T(n)$ depends on $T$, not only on the number of vertices of $T$. Indeed, we saw that if $T$ is a path, then $C_T(n)$ is roughly $n^{k/3}$. At the same time, if $T$ is a star with $k$ leaves then $C_T(n)=\Theta(n^k)$. To see this, fix the centre of the star and distribute the remaining points equally on concentric circles of radii $\delta_1,\ldots, \delta_k$  around the centre of the star. One can similarly find examples in for $d=3$ where $C_T^3(n)$ depends on the tree itself.

For some trees determining $C_T(n)$ trivially reduces to determining $C_k(n)$ for some $k$. However, for many other trees the problem seems challenging, and new ideas are needed to tackle it. Subdivisions of stars show that it in some cases it might not be possible to determine $C_T(n)$ without knowing $u_2(n)$, even in terms of $u_2(n)$. To see this, let $T_{\ell,3}$ be a star-shaped tree on $3\ell+1$ vertices, with one (central) vertex of degree $\ell$, and $\ell$ paths on $3$ vertices joined to the central vertex. (This tree for $\ell =3$ is depicted on Figure \ref{fig2}, right.) 
Generalising the lower bound constructions we used for the chains in two different ways, we can obtain $C_{T_{\ell,3}}(n)=\Omega(u_2(n)^{\ell})$ (by fixing the central vertex of the tree) and $C_{T_{\ell,3}}(n)=\Omega(n^{\ell+1})$ (by fixing all vertices that are neighbours of the leaves). This is illustrated on Figure \ref{fig1} for $\ell=3$.  It would be interesting to prove that $C_{T_{\ell,3}(n)}$ is the maximum of these two lower bounds. 

\begin{figure}[h]
\centering
{\includegraphics[width=10cm]{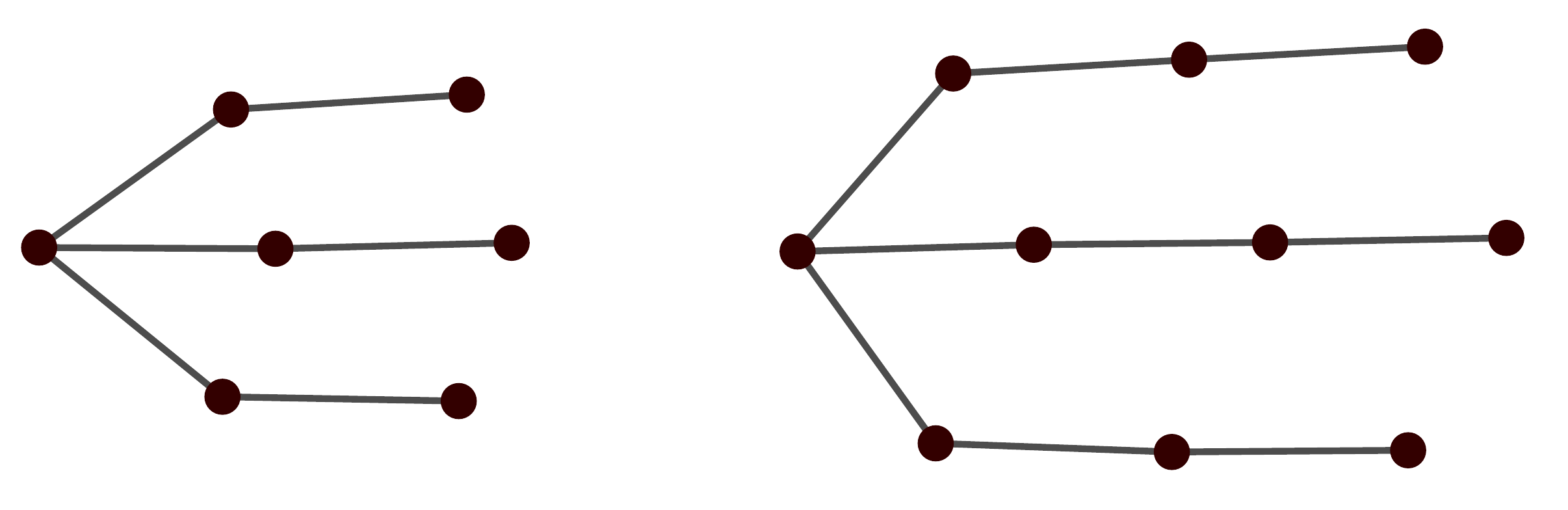}}
\caption{Star-shaped trees.}\label{fig2}
\end{figure}

\begin{figure}[h]
\centering
{\includegraphics[width=16.5cm]{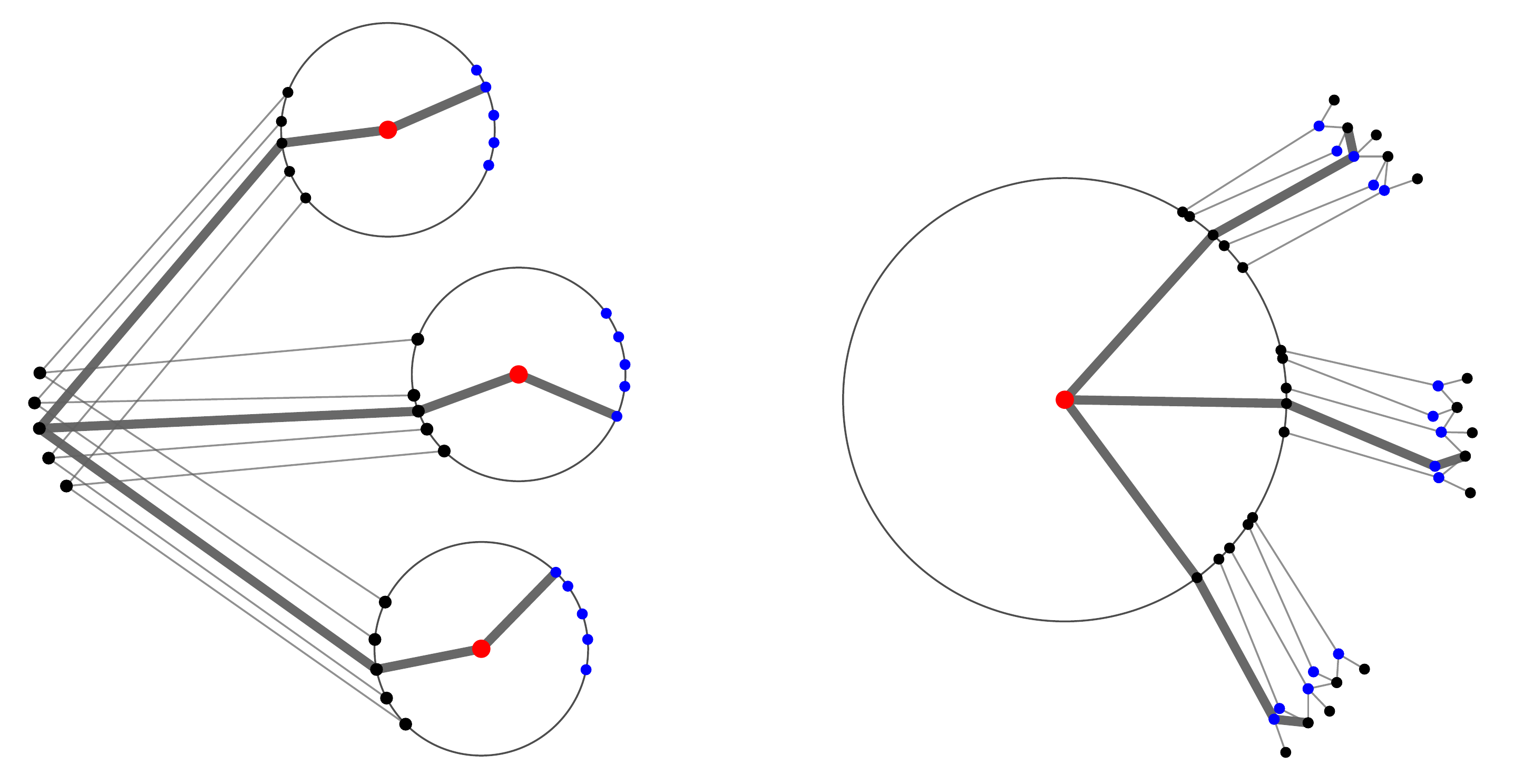}}
\caption{Examples providing lower bounds for the number of copies of $T_{\ell,3}$ with $\ell=3$. Left: $\Omega(n^{\ell+1})$ copies of  $T_{\ell,3}$. Right: $\Omega(u_2^{\ell})$ copies of  $T_{\ell,3}$. Vertices of red colour indicate parts (corresponding to the vertices of the tree and) that consist of a single vertex.}\label{fig1}
\end{figure}

\begin{pro}\label{pro1}Is it true that $C_{T_{\ell,3}}(n)=\Theta(\max\{n^{\ell+1},u_2(n)^{\ell}\})$?
\end{pro}

Motivated by the constructions described before, we  propose the following more general question.

\begin{pro}\label{treeconj}Is it true that for every tree $T$ there are integers $m,\ell$ such that $C_T(n)=\Theta(n^mu_2(n)^\ell)$?
\end{pro}

The smallest tree $T$, for which we cannot determine the order of magnitude of $C_T(n)$, is a star-shaped tree on $7$ vertices with one central vertex of degree $3$ and three and three paths on $2$ vertices joined to the central vertex (see the left tree of Figure \ref{fig2}).

\begin{pro}\label{small}
Is it true that for the star-shaped tree $T$ on $7$ vertices described above we have $C_T(n)=\Theta(n^3)$?
\end{pro}

Note that it may be easier (and also very interesting) to obtain upper bounds in Problems \ref{pro1}-\ref{small}  with a poly-logarithmic or $n^{\varepsilon}$ error term, as we did in Theorem \ref{main}.

In a follow-up paper, we find almost sharp bounds for $C_T(n)$ for some `non-trivial' trees (i.e., those that cannot be reduced to the chains case). Further, we will describe a general lower bound construction that motivates Problem \ref{treeconj}. We will also provide some partial results for Problem \ref{small} and connect it to an interesting incidence problem. 

Finally, it would be interesting to decide if the lower bounds in Proposition \ref{oddk} are sharp for $\bm\delta=(1,\dots,1)$. The first open case is $k=3$.

\begin{pro}
Is it true that $C_3^3(n)=\Theta\left (\max\left\{ \frac{u_3(n)^3}{n^2},us_3(n)n\right \}\right )$ for $\bm\delta=(1,\dots,1)$?
\end{pro}

\subsection*{Acknowledgements.} We thank Konrad Swanepoel and Peter Allen for helpful comments on the manuscript. We also thank Dömötör Pálvölgyi for suggesting Proposition \ref{Propprob}. The authors acknowledge the financial support from the Russian Government in the framework of MegaGrant no 075-15-2019-1926.

\bibliographystyle{amsplain}
\bibliography{biblio}

\providecommand{\bysame}{\leavevmode\hbox to3em{\hrulefill}\thinspace}
\providecommand{\MR}{\relax\ifhmode\unskip\space\fi MR }
\providecommand{\MRhref}[2]{%
  \href{http://www.ams.org/mathscinet-getitem?mr=#1}{#2}
}
\providecommand{\href}[2]{#2}
\begin{thebibliography}{10}

\bibitem{us2}
P.~K. Agarwal, E.~Nevo, J.~Pach, R.~Pinchasi, M.~Sharir, and S.~Smorodinsky,
  \emph{Lenses in arrangements of pseudo-circles and their applications}, J.
  ACM \textbf{51} (2004), no.~2, 139--186.

\bibitem{AS}
B.~Aronov and M.~Sharir, \emph{Cutting circles into pseudo-segments and
  improved bounds for incidences}, Discrete Comput. Geom. \textbf{28} (2002),
  no.~4, 475--490.

\bibitem{brass}
P.~Brass, \emph{On the maximum number of unit distances among n points in
  dimension four}, Intuitive Geometry \textbf{6} (1997), 277--290.

\bibitem{erdos}
P.~Erd\H{o}s, \emph{On sets of distances of $n$ points}, Amer. Math. Monthly
  \textbf{53} (1946), no.~5, 248--250.

\bibitem{erdos3}
\bysame, \emph{On sets of distances of $n$ points in {E}uclidean space}, Magyar
  Tud. Akad. Mat. Kutato Int. Közl. \textbf{5} (1960), 165--169.

\bibitem{GK}
L.~Guth and N.~H. Katz, \emph{On the {Erdős} distinct distances problem in the
  plane}, Ann. of Math. \textbf{181} (2015), no.~1, 155--190.

\bibitem{KMSS}
H.~Kaplan, J.~Matou{\v{s}}ek, Z.~Safernov{\'a}, and M.~Sharir, \emph{Unit
  distances in three dimensions}, Combin. Probab. Comput. \textbf{21} (2012),
  no.~4, 597--610.

\bibitem{us3}
A.~Marcus and G.~Tardos, \emph{Intersection reverse sequences and geometric
  applications}, J. Combin. Theory Ser. A \textbf{113} (2006), no.~4, 675--691.

\bibitem{us1}
J.~Pach and M.~Sharir, \emph{Geometric incidences}, Towards a theory of
  geometric graphs, Contemp. Math., vol. 342, Amer. Math. Soc., Providence, RI,
  2004, pp.~185--223.

\bibitem{Shef}
E.~A. Palsson, S.~Senger, and A.~Sheffer, \emph{On the number of discrete
  chains}, 2019, arXiv:1902.08259.

\bibitem{SST}
J.~Spencer, E.~Szemer{\'e}di, and W.~T Trotter, \emph{Unit distances in the
  {E}uclidean plane}, Graph theory and combinatorics, Academic Press, 1984,
  pp.~294--304.

\bibitem{Sw}
K.~J. Swanepoel, \emph{Unit distances and diameters in euclidean spaces},
  Discrete Comput. Geom. \textbf{41} (2009), no.~1, 1--27.

\bibitem{Za2}
J.~Zahl, \emph{An improved bound on the number of point-surface incidences in
  three dimensions}, Contrib. Discrete Math. \textbf{8} (2013), no.~1,
  100--121.

\bibitem{Za}
\bysame, \emph{Breaking the 3/2 barrier for unit distances in three
  dimensions}, Int. Math. Res. Not. \textbf{2019} (2019), no.~20, 6235--6284.

\end{thebibliography}

\end{document}